\documentclass[11pt]{amsart}

\usepackage{amsmath}
\usepackage{amssymb}
\usepackage{amsthm, bm}
\usepackage[all,cmtip]{xy}
\usepackage{color}
\usepackage{comment}

\theoremstyle{plain}
\newtheorem{thm}{Theorem}[section]
\newtheorem{prop}[thm]{Proposition}
\newtheorem{lem}[thm]{Lemma}
\newtheorem{cor}[thm]{Corollary}

\theoremstyle{definition}
\newtheorem{defn}[thm]{Definition}

\theoremstyle{remark}
\newtheorem{rem}[thm]{Remark}
\newtheorem{claim}[thm]{Claim}

\newcommand{\bC}{\ensuremath{\mathbb{C}}}

\newcommand{\bP}{\ensuremath{\mathbb{P}}}

\newcommand{\bZ}{\ensuremath{\mathbb{Z}}}

\newcommand{\cC}{\ensuremath{\mathcal{C}}}

\newcommand{\cH}{\ensuremath{\mathcal{H}}}

\newcommand{\cL}{\ensuremath{\mathcal{L}}}
\newcommand{\cM}{\ensuremath{\mathcal{M}}}
\newcommand{\cN}{\ensuremath{\mathcal{N}}}
\newcommand{\cO}{\ensuremath{\mathcal{O}}}

\newcommand{\cX}{\ensuremath{\mathcal{X}}}

\DeclareMathOperator{\Pic}{Pic}
\DeclareMathOperator{\Sing}{Sing}

\DeclareMathOperator{\Nef}{Nef}
\DeclareMathOperator{\Aut}{Aut}

\DeclareMathOperator{\Gal}{Gal}
\DeclareMathOperator{\mult}{mult}

\DeclareMathOperator{\ord}{ord}

\DeclareMathOperator{\rk}{rk}

\begin{document}

\title
[Non-K\"{a}hler Calabi-Yau $3$-folds]
{Examples of non-K\"{a}hler Calabi-Yau $3$-folds 
with arbitrarily large $b_2$}

\author{Kenji Hashimoto}
\address{Graduate School of Mathematical Sciences, The University of Tokyo,
3-8-1 Komaba, Meguro-ku, Tokyo 153-8914, Japan }
\email{hashi@ms.u-tokyo.ac.jp}

\author{Taro Sano}
%\address{Max Planck Institute for Mathematics
%\newline Vivatsgasse 7, 53111 Bonn, Germany}
\address{ Department of Mathematics, Graduate School of Science, Kobe university, 
\newline 1-1, Rokkodai, Nada-ku, Kobe 657-8501, Japan}
\email{tarosano@math.kobe-u.ac.jp}

\maketitle

\begin{abstract}
We construct non-K\"{a}hler simply connected Calabi-Yau 3-folds with arbitrarily large 2nd Betti numbers by smoothing 
normal crossing varieties with trivial dualizing sheaves. 
\end{abstract}

\tableofcontents

\section{Introduction} 
In this paper, a {\it Calabi-Yau manifold} means a compact complex manifold whose canonical bundle is trivial 
and $H^i(X, \cO_X)= H^0(X, \Omega^i_X) = 0$ for $0 <i < \dim X$. A projective Calabi-Yau manifold is often also called a strict Calabi-Yau manifold.  
Our main interest in this paper is a {\it Calabi-Yau 3-fold}, that is, a Calabi-Yau manifold of dimension $3$. 

Projective Calabi-Yau manifolds are one of the building blocks in the classification of algebraic varieties. 
Nevertheless, it is not known whether there are only finitely many topological types of projective Calabi-Yau 3-folds or not. 
The main purpose of this paper is to give infinitely many topological types of non-K\"{a}hler Calabi-Yau $3$-folds as follows. 

\begin{thm}\label{thm:mainintroduction}
Let $a>0$ be any positive integer. 
Then there exists a simply connected Calabi-Yau 3-fold $X(a)$ with the 2nd Betti number $b_2(X(a)) = a+3$, the topological Euler number $e(X(a)) = -256a^2+32a -224$ and the algebraic dimension $a(X(a)) =1$. 
\end{thm}

%By a small modification of our construction, we can also construct examples with $b_2(X) =a$, so the particular number $8a^2 +3$ is not important (See Example \ref{eg:arbitraryb2}). 
As far as we know, our examples are the first examples of 
%(non-K\"{a}hler but) 
complex Calabi-Yau 3-folds with arbitrarily large $b_2$ in our sense. Their topological Euler number can be arbitrarily small negative. 
For positive integers $a \neq a'$, we see that $X(a)$ and $X(a')$ are not bimeromorphic (Remark \ref{rem:nonbimeromorphic}), thus show bimeromorphic unboundedness of non-K\"{a}hler Calabi-Yau 3-folds. It is also remarkable that the Hodge to de Rham spectral sequence degenerates at $E_1$ on $X(a)$ and they have unobstructed deformations (Remark \ref{rem:HodgeSymmetry}).  

Friedman \cite[Example 8.9]{MR1141199} constructed infinitely many topological types of Calabi-Yau $3$-folds with $b_2=0$ by deforming a Calabi-Yau $3$-fold with ordinary double points based on Clemens' construction of some quintic 3-folds with infinitely many smooth rational curves \cite{MR720930}. 
There are also infinitely many examples of Calabi-Yau 3-folds of $b_2 =1$ but with different cubic forms on $H^2$
as flops of a fixed Calabi-Yau $3$-fold (cf.\ \cite[Example 7.6]{MR1141199}, \cite[Example 14]{MR1365849}). Fine--Panov  \cite[Section 3]{MR2679581} constructed simply connected compact complex 3-folds with trivial canonical bundle, 
arbitrarily large $b_2$ and non-zero holomorphic 2-forms.

Non-K\"{a}hler Calabi-Yau manifolds are also interesting from the point of view of string theory (the Strominger equations) and complex differential geometry (cf.  \cite{MR2891478}, \cite{MR2985333}, \cite{MR3372471} and references therein). 

We shall construct the examples by smoothing simple normal crossing (SNC) varieties via the log deformation theory developed by Kawamata--Namikawa \cite{MR1296351}. 
Lee \cite{MR2658406} considered log deformations of SNC varieties consisting of two irreducible components which are called Tyurin degenerations. 
We also use Tyurin degenerations to construct our examples.
The new point in this paper is to consider gluing automorphisms of the intersection of irreducible components of SNC varieties. 
%It seems that Tyurin expected that there are only finitely many topological types of projective Calabi-Yau 3-folds which admit Tyurin degenerations \cite[Introduction]{MR2112600}. 
Tyurin degenerations are also studied in the context of mirror symmetry (\cite{MR2112600}, \cite{MR3751815}, \cite{MR3633789}). 

%As a bi-product of the construction, we obtain an example of involutions of K3 surfaces in a family 
%which are induced from a birational involution of the total space of the family (Section \ref{section:specialK3}). 
%It turns out that the birational involution is a flop (Claim \ref{claim:involutionverygeneral}) and indeterminate along flopping curves. 

\subsection{Sketch of the construction} 

First, we prepare an SNC variety $X_0(a)=X_1 \cup X_2$, where $X_1$ is the blow-up of $\bP^1 \times \bP^1 \times \bP^1$ along some curves 
$f_1, \ldots, f_{a}, C$ and 
$X_2:=\bP^1 \times \bP^1 \times \bP^1$. The curves $f_1, \ldots, f_{a}$ are distinct smooth fibers of an elliptic fibration 
$S \rightarrow \bP^1$ on a very general $(2,2,2)$-hypersurface $S$ induced by the 1st projection. 
We glue $X_1$ and $X_2$ along $S$ and its strict transform to construct $X_0(a)$. 
Since $S$ is an anticanonical member, we have $\omega_{X_0(a)} \simeq \cO_{X_0(a)}$. 
In order to make $X_0(a)$ ``$d$-semistable'', we need to blow-up $f_1, \ldots , f_{a}$ and some curve $C$. 
The point is that we glue after twisting by a certain automorphism of $S$ of infinite order. 
Because of this, the number of blow-up centers for $X_1$ can be 
arbitrarily large.   

Thus we obtain $X_0(a)$ which satisfies the hypothesis of Theorem \ref{thm:kawamata-namikawa} \cite[Theorem 4.2]{MR1296351} 
and can deform $X_0(a)$ to a Calabi-Yau $3$-fold $X(a)$ 
which turns out to be non-K\"{a}hler. 
This $X(a)$ is the example in Theorem \ref{thm:mainintroduction}. 
Note that we can apply the smoothing result even when the SNC variety itself is not projective but the irreducible components are K\"{a}hler. 
(See Remark \ref{rem:kawamata-namikawa})

We check that $X_0(a)$ and $X(a)$ are both non-K\"{a}hler if we twist by a non-trivial automorphism of $S$ (Proposition \ref{prop:nonprojective}. See also Remark \ref{rem:classC}). 
We use Lemma \ref{lem:biglinebundle} which states that, under some conditions, an SNC variety which is a degeneration of a projective Calabi-Yau manifold 
admits a big line bundle whose restriction to each irreducible component still has a non-zero section. 
Moreover, we show that the algebraic dimension of $X$ is $1$ (Proposition \ref{prop:algdimX1}). 
We also compute the topological Euler number of $X$ (Claim \ref{claim:Eulernumber}) and check that $X$ is simply connected (Proposition \ref{prop:simplyconnected}).

\subsection{Notations} We work over the complex number field $\bC$ throughout the paper. 
We call a complex analytic space $X$ {\it a (proper) SNC variety} if $X$ has only normal crossing singularities and 
its irreducible components are smooth (proper) varieties. 
We identify a proper scheme over $\bC$ and its associated compact analytic space unless otherwise stated. 

Let $X$ be a proper SNC variety and $\phi \colon \cX \rightarrow \Delta^1$ be a proper flat morphism of analytic spaces over a unit disk $\Delta^1$ 
such that $\phi^{-1}(0) \simeq X$, that is, $\phi$ is {\it a deformation of $X$}. 
We call $\phi$ {\it a semistable smoothing} of $X$ if $\cX$ is smooth and its general fiber $\cX_t:= \phi^{-1}(t)$ is smooth for $t \neq 0$. 

\section{Preliminaries}

%The results in this section is used in the construction of the examples. 
%The reader can skip this section and go back when the results are needed. 

The following result guarantees the existence of a gluing of two schemes along their isomorphic closed subschemes. 

\begin{thm}\label{thm:coproduct}( \cite[1.1]{MR0335524}, \cite[Th\'{e}or\`{e}me 5.4, Th\'{e}or\`{e}me 7.1]{MR2044495}, \cite[Corollary 3.9]{MR2182775})
Let $Y, X_1, X_2$ be schemes and $\iota_i \colon Y \hookrightarrow X_i$ be  closed immersions for $i=1,2$. 
Then there exists a scheme $X$ in the Cartesian diagram
\[
\xymatrix{
Y \ar@{^{(}->}[d]^{\iota_2} \ar@{^{(}->}[r]^{\iota_1} & X_1  \ar@{^{(}->}[d]^{\phi_1} \\
X_2 \ar@{^{(}->}^{\phi_2}[r] & X
}
\]
such that  $\phi_1$ and $\phi_2$ are  closed immersions and they  
induce isomorphisms $X_1 \setminus Y \xrightarrow{\simeq} X \setminus X_2$ and $X_2 \setminus Y \simeq X \setminus X_1$. 
(We say that $X$ is the push-out of the morphisms $\iota_1$ and $\iota_2$.)
\end{thm}

\begin{proof}
We can show the proposition following \cite[1.1]{MR0335524}. 
\end{proof}

By this, we have the following corollary. 

\begin{cor}\label{cor:pushoutSNC} 
Let $X_1, X_2$ be  smooth proper varieties and $D_i \subset X_i$ be smooth divisors for $i=1,2$ 
with an isomorphism $\phi \colon D_1 \xrightarrow{\simeq} D_2$. Let $i_1 \colon D_1 \hookrightarrow X_1$ and $i_2 \colon D_2 \hookrightarrow X_2$ be 
the given closed immersions and let 
$Y$ be the push-out of two closed immersions $\iota_1:= i_1$ and $\iota_2:= i_2 \circ \phi$ which exists by Theorem \ref{thm:coproduct}. 

Then $Y$ is a proper SNC variety with two irreducible components $Y_1$ and $Y_2$ such that $Y_i \simeq X_i$ and 
$Y_1 \cap Y_2 \simeq D_i$ for $i=1,2$.  (We denote by $Y =: Y_1 \cup^{\phi} Y_2$ the push-out.)
\end{cor}

\begin{proof} 
We check the properness of $Y$ by definition of properness. 
We also check that $Y$ is normal crossing by a local computation. 
\end{proof}

\begin{rem}
Note that a proper SNC variety is non-projective in general 
even if its irreducible components are projective. Let $X= X_1 \cup X_2$ be a proper SNC variety such that $X_1$ and $X_2$ 
are projective varieties and $D:= X_1 \cap X_2$. Then $X$ is projective if and only if there are ample line bundles $\cL_1$ on $X_1$ and $\cL_2$ on $X_2$ 
such that $\cL_1|_D \simeq \cL_2|_D$. 
\end{rem}

\begin{defn}\label{defn:d-semistable}
Let $X$ be an SNC variety and $X= \bigcup_{i=1}^N X_i$ be the decomposition into its irreducible components. Let $D:= \Sing X = \bigcup_{i \neq j} (X_i \cap X_j)$ be the double locus and let  $I_{X_i}, I_D \subset \cO_X$ be the ideal sheaves of 
$X_i$ and $D$ on $X$. Let 
\[
\cO_D(X):= (\bigotimes_{i=1}^N I_{X_i} / I_{X_i} I_D)^{\ast} \in \Pic D
\]  
be the {\it infinitesimal normal bundle}  as in \cite[Definition 1.9]{MR707162}. 

We say that $X$ is {\it $d$-semistable} if $\cO_D(X) \simeq \cO_D$. 
\end{defn}

\begin{rem}
Let $X =  \bigcup_{i=1}^N X_i$ and $D$ be as in Definition \ref{defn:d-semistable}.  
Friedman \cite[Corollary 1.12]{MR707162} proved that, if $X$ has a semistable smoothing, then $X$ is $d$-semistable. 

When $N=2$, we have $\cO_D(X) \simeq \cN_{D/X_1} \otimes_{\cO_{D}} \cN_{D/X_2}$ 
via a suitable identification of $D \subset X_1$ and $D \subset X_2$, where $\cN_{D/X_i}$ is the normal bundle of $D \subset X_i$ for $i=1,2$.   
\end{rem}

We use the following theorem of Kawamata--Namikawa. 

\begin{thm}\label{thm:kawamata-namikawa}(\cite[Theorem 4.2]{MR1296351}, \cite[Corollary 7.4]{CLM})
Let $n \ge 3$ and 
$X$ be an $n$-dimensional proper SNC variety which satisfies the following;  
\begin{enumerate}
\item $\omega_X \simeq \cO_X$. 
\item $H^{n-1}(X, \cO_X) =0$, $H^{n-2}(X^{\nu}, \cO_{X^{\nu}}) =0$, where  $X^{\nu} \rightarrow X$ is the normalization. %(\xr{We don't need this by a recent result}). 
\item $X$ is $d$-semistable. 
\end{enumerate}
Then there exists a semistable smoothing $\phi \colon \cX \rightarrow \Delta^1$ of $X$ over a unit disk.  
\end{thm}

\begin{rem}\label{rem:kawamata-namikawa} 
 In \cite[Theorem 4.2]{MR1296351}, it is assumed that $X$ is K\"{a}hler. However, we check that we only need to assume that $X$ is a proper SNC variety (or each irreducible component is  
K\"{a}hler). 

The essential tools are two spectral sequences. 
One is the log Hodge to de Rham spectral sequence in \cite[Lemma 4.1]{MR1296351} which follows from the construction of a cohomological mixed Hodge complex (See also \cite[Theorem 3.12, Proposition 3.19]{MR2033010}). 
Another spectral sequence is the one in \cite[Proposition 1.5 (3)]{MR707162} which 
uses only the existence of a pure Hodge structure on the stratum of an SNC variety.  

By using these spectral sequences, we show that log deformations of $X$ are unobstructed as in \cite[Theorem 4.2]{MR1296351}. 
Then, by using the existence of the Kuranishi family of $X$ and Artin's approximation, we can construct a semistable smoothing $\phi \colon \cX \rightarrow \Delta^1$ of $X$ (cf.\ \cite[Corollary 2.4]{MR1296351}).  

However, if $X$ is not projective, the general fiber of $\phi$ may not be an algebraic variety even when $H^2(X, \cO_X) =0$. 
Indeed, this happens in the examples in Theorem \ref{thm:mainexample}. 

 In \cite[Corollary 7.4]{CLM}, the above result is shown without the assumption (2) for a projective SNC variety. It is quite possible that we can remove the assumption when $X$ is not projective but proper (cf. \cite{Felten:2019aa}, \cite{Felten:2020aa}). 
\end{rem}

\begin{rem}\label{rem: dualizingtrivial}
Let $X$ be a proper SNC variety such that $X = X_1 \cup X_2$. If its dualizing sheaf $\omega_X \simeq \cO_X$, then $D:= X_1 \cap X_2$ should satisfy 
$D \in |\omega_{X_i}^{-1}|$. The converse does not hold in general since $D$ may not be connected. 
However, if $D$ is connected and $D \in  |\omega_{X_i}^{-1}|$, 
then we check that $\omega_X \simeq \cO_X$. 
\end{rem}

In order to study a general fiber of a smoothing of an SNC variety, the following map of Clemens is useful.  

\begin{thm}\label{fact:Clemenscontraction}(cf.\ \cite[Theorem 6.9]{MR0444662}, \cite[Theorems 5.2, 5.4]{MR1808639})
Let $\phi \colon \cX \rightarrow \Delta$ be a proper flat surjective morphism from a complex manifold $\cX$ onto a $1$-dimensional open disk $\Delta$. 
Assume that $\phi$ is smooth over $\Delta \setminus {0}$ and $\phi^{-1}(0) \subset \cX$ is an SNC divisor. Let $\cX_t:= \phi^{-1}(t)$ and, for $k \ge 0$, 
let $\cX_0^{[k]} \subset \cX_0$ be the locus where $k+1$ irreducible components of $\cX_0$ intersect. 

Then we have a continuous map $c_t \colon \cX_t \rightarrow \cX_0$ for $t \neq 0$ 
 such that we have a homeomorphism $c_t^{-1}(p) \approx (S^1)^k$ when $p \in \cX_0^{[k]} \setminus \cX_0^{[k+1]}$ and 
$c_t$ induces a homeomorphism $c_t^{-1}(\cX_0 \setminus \cX_0^{[1]}) \xrightarrow{\approx} \cX_0 \setminus \cX_0^{[1]}$.  
(We call the map $c_t$ the Clemens contraction. )
\end{thm}

\section{Construction of non-K\"{a}hler Calabi-Yau 3-folds with arbitrarily large $b_2$}	

\subsection{Construction of the examples}

First we explain the K3 surface which is essential in the construction of our Calabi-Yau 3-folds.  

Let 
\[
P(3):= \bP^1 \times \bP^1 \times \bP^1
\] and  
$\bP_i:= \bP^1$ be the $i$-th factor of $P(3)$ for $i=1,2,3$. 
Let $\pi_i \colon P(3) \rightarrow \bP_i$ be the $i$-th projection   
and $\cO_{P(3)}(c_1, c_2, c_3):= \bigotimes_{i=1}^3 \pi_i^*\cO_{\bP_i}(c_i)$ a line bundle on $P(3)$ for $c_1, c_2, c_3 \in \bZ$.  

Let $S \subset P(3)$ be a very general $(2,2,2)$-hypersurface, that is, 
a very general element of the linear system $|\cO_{P(3)}(2,2,2)|$.  
Then $S$ is a K3 surface. 
This surface is called a {\it Wehler surface} and studied in several articles (cf.\ \cite{MR1352278}, \cite{MR2828530}, \cite{MR3372313}). 
We shall recall some of its properties. 

 For $1\le i < j \le 3$, the surface $S$ has a covering involution  
$\iota_{ij} \colon S \rightarrow S$ corresponding to the double cover 
$p_{ij} \colon S \rightarrow \bP_i \times \bP_j$ which is induced by 
the projection $\pi_{ij} \colon P(3) \rightarrow \bP_i \times \bP_j$. 
By the Noether-Lefschetz theorem (cf.\ \cite[Proposition 2.27, Theorem 3.33]{MR2449178} or \cite[Theorem 1]{MR2567426}), we see that $\Pic S = \bZ e_1 \oplus \bZ e_2 \oplus \bZ e_3$, where $e_i$ is the fiber class of the elliptic fibration 
$p_i \colon S \rightarrow \bP_i$ for $i=1,2,3$. 
By this, we see that $S$ contains no $(-2)$-curve. 
Indeed, for $D= \sum_{i=1}^3 a_i e_i \in \Pic S$, we have 
\[
D^2 = 2 (\sum_{1 \le i < j \le 3} a_i a_j (e_i \cdot e_j)) =4 (a_1 a_2 + a_1 a_3 + a_2 a_3) \in 4 \bZ
\] by $e_i^2 =0$ and $e_i \cdot e_j =2$ for  $i \neq j$.   	
Hence the nef cone $\Nef S \subset \Pic S$ of $S$ can be described as the positive cone 
\[
\Nef S = \{D \in \Pic S \mid D^2 \ge 0, D \cdot H \ge 0 \},  
\]
where $H:= e_1 + e_2 +e_3$. 
First we need the following claim on the action of the involution $\iota_{ij}$ on $\Pic S$. 

\begin{claim}\label{claim:involutionverygeneral} (cf.\ \cite[Lemma 2.1]{MR1352278})
 Let $i,j,k \in \bZ$ be integers such that $i<j$ and $\{i,j,k \} = \{1,2,3 \}$. Then we have the following.
\begin{enumerate}
\item[(i)] $\iota_{ij}^*(e_i) = e_i$ and $\iota_{ij}^*(e_j) = e_j$,
\item[(ii)] $\iota_{ij}^*(e_k) = 2e_i + 2e_j - e_k$. 
\end{enumerate}
\end{claim}

\begin{proof}[Proof of Claim]  
(i) This follows since $\iota_{ij}$ interchanges 
two points in a fiber $p_{ij}^{-1}(p)$ for a general $p \in \bP_i \times \bP_j$. 

\noindent(ii) 
We shall show that $\iota_{12}^*(e_3) = 2e_1 + 2 e_2 -e_3$. The others can be shown similarly. 

Let  $[S_0: S_1], [T_0: T_1], [U_0: U_1]$ be the coordinates of $\bP_1, \bP_2, \bP_3$. 
Note that $S$ can be described as 
\[
S=   (U_0^2 F_1 + U_1^2 F_2 + U_0 U_1 F_3 =0)  \subset P(3) 
\]
for some very general (2,2)-polynomials $F_1, F_2, F_3 \in H^0(\bP_1 \times \bP_2, \cO(2,2))$ on $\bP_{1} \times \bP_2 = \bP^1 \times \bP^1$.  
Let $s:=S_1/ S_0, t:= T_1/T_0, u:= U_1/ U_0$ be the affine coordinates of the $\bP_1, \bP_2, \bP_3$ respectively. 
Also, for $i=1,2,3$, let $f_i \in \bC[s,t]$ be the dehomogenization of $F_i$.   
Then the function field $K(S)$ of $S$ can be described as 
\[
K(S) \simeq \bC(s,t)[u]/(u^2 + \frac{f_3}{f_2} u + \frac{f_1}{f_2}) 
\]
and $\iota_{12}$ induces an element $\iota_{12}^{\sharp} \in \Gal(K(S)/ K(\bP_1 \times \bP_2))$  determined by 
\[
\iota_{12}^{\sharp}(u) = -u - \frac{f_3}{f_2}.
\]  
%In projective coordinates, $\iota_{12}$ of $p_{12}$ is determined by 
%\[
%[S_0 : S_1] \mapsto [S_0 : S_1], \ \ [T_0 : T_1] \mapsto [T_0 : T_1], 
%\]
%\[
%[U_0 : U_1] \mapsto [F_2U_0 : - F_2 U_1 - {F_3} U_0] = [- F_1 U_0 - {F_3} U_1 : F_1 U_1]. 
%\]
%Indeed, it is defined outside the finite points $p_{13}^{-1}(F_1 = F_2 =0)$ and this extends to a biregular involution of a K3 surface $S$. 
%, since $F_1, F_2, F_3$ are very general and we have $F_3 \neq 0$ on $(F_1 = F_2 =0)$, we can extend the involution on those points as well. 

%Note that $p_{13}^*(p_{13})_*(e_2) = e_2 + \iota_{13}^* (e_2)$ and $e_1 = p_{13}^*(e'_1), e_3 = p_{13}^*(e'_3)$, where $e'_j$ is a fiber of the projection $q_j \colon \bP^1 \times \bP^1 \rightarrow \bP^1$ for $j=1,3$ which fits in 
%\[S \xrightarrow{p_{ij}}  \bP^1 \times \bP^1 \xrightarrow{q_j} \bP^1. \]Since we have $(p_{13})_* (e_2) \cdot e'_j =2$ for $j=1,3$ by the projection formula, we see that $(p_{13})_* (e_2)  = 2(e'_1 + e'_3)$. 

By this description of $\iota_{12}^{\sharp}$, we see that 
\begin{equation}\label{eq:iota12action}
\iota_{12}^*(e_3) \neq e_3
\end{equation} since $F_1, F_2, F_3$ are very general and 
$\frac{f_3}{f_2}$ is not constant on the fiber of $p_3 \colon S \rightarrow \bP_3$.  By (\ref{eq:iota12action}) and 
\[
\iota_{12}^*(e_1) \cdot \iota_{12}^*(e_3) = \iota_{12}^*(e_2) \cdot \iota_{12}^*(e_3) = 2, \ \ 
\iota_{12}^*(e_3)^2 =0,
\] 
we check that 
$\iota_{12}^*(e_3) = 2e_1 + 2e_2 - e_3$.  
\end{proof}

Now let $\iota := \iota _{12} \circ \iota_{13}$, that is,  
\[ \iota  \colon S \xrightarrow{\iota_{13}} S \xrightarrow{\iota_{12}} S.
\] 
\begin{claim}\label{claim:autoactionmatrix}
The automorphism $\iota$ induces the linear automorphism $\iota^* \in \Aut(\Pic S)$ 
corresponding to a matrix 
\[
\begin{pmatrix}
1 & 2 & 6 \\
0 & -1 & -2 \\
0 & 2 & 3
\end{pmatrix}, 
\]
that is, we have $\iota^*(e_1) = e_1, \iota^*(e_2) = 2e_1 -e_2 +2 e_3, 
\iota^* (e_3) = 6e_1 -2e_2 +3e_3$. 
\end{claim}

\begin{proof}[Proof of Claim]

Since we have $\iota_{12}^*(e_1) = e_1, \iota_{13}^*(e_1) = e_1$, 
we obtain $\iota^*(e_1) = e_1$. 

We have $\iota_{13}^*(e_2) = 2 e_1 - e_2 + 2e_3$ by Claim \ref{claim:involutionverygeneral}.  
By this and $\iota_{12}^*(e_2) = e_2$, we obtain $\iota^* (e_2) = 2e_1 - e_2 + 2e_3$. 

By a similar computation, we obtain $\iota^* (e_3) = 6e_1 -2e_2 +3e_3$. 
Indeed, we have $\iota_{12}^*(e_3) = 2 e_1 + 2 e_2 -e_3$ and 
$\iota_{13}^*(2e_1 + 2 e_2 - e_3) = 
2e_1 + 2(2e_1 - e_2 + 2 e_3) - e_3 = 6e_1 - 2 e_2 + 3 e_3$. 
\end{proof}

By this claim, for $a \in \bZ$, the $a$-th power $\iota^a \in \Aut S$ of $\iota$ induces 
\[
(\iota^a)^* = 
\begin{pmatrix}
1 & 4a^2 -2a & 4a^2 +2a \\
0 & 1-2a & -2a \\
0 & 2a & 1+2a
\end{pmatrix} \in \Aut (\bZ^3) \simeq \Aut (\Pic S) 
\]	
with respect to the basis $e_1, e_2, e_3 \in \Pic S$ (by induction on $a$ or use JCF). 

\begin{rem}
In \cite[3.4]{MR3372313}, $\Aut (S)$ is studied in detail. They show that $\Aut (S)$ is a free product of 3 cyclic groups of order 2 generated by 
the three involutions $\iota_{12}, \iota_{13}, \iota_{23}$. 
\end{rem}

Now we can construct our Calabi-Yau 3-folds as follows. 

\begin{thm}\label{thm:mainexample}
Let $a \in \bZ$ be a positive integer. 
Then there exists a Calabi-Yau 3-fold $X:=X(a)$ such that 
$b_2(X) = a+3$ and $e(X) = -256 a^2 +32a-224$, where $b_2(X)$ is the 2nd Betti number of $X$ and $e(X)$ is the topological Euler number of $X$. 
\end{thm}

\begin{proof}
We first construct an SNC variety $X_0(a)$ by gluing two smooth projective varieties $X_1$ and $X_2$ as follows. 
For $c_1, c_2, c_3 \in \bZ$, we let $\cO_{S}(c_1, c_2, c_3):= \cO_{P(3)}(c_1, c_2, c_3)|_S$. 

\noindent{\bf (Construction of $X_1$ and $X_2$)}  
Let $\mu' \colon X'_1 \rightarrow P(3)$ be the blow-up of 
$f_1, \ldots, f_{a}$.  
where $f_1, \ldots, f_{a} \in |\cO_S(1,0,0)|$ are disjoint smooth fibers of 
the elliptic fibration $p_1 \colon S \rightarrow \bP^1$. 
Let $\nu \colon X_1 \rightarrow X'_1$ be the blow-up along the strict transform of $C_a$, 
where $C_a \in |\cO_S(16a^2-a+4, 4-8a, 4+8a)|$ is a general smooth member. 
Note that $\cO_S(16a^2-a+4, 4-8a, 4+8a)$ is ample since we have 
\begin{multline*}
\cO_S(16a^2-a+4, 4-8a, 4+8a)^2 = 4 (8 (16a^2-a +4) + (4-8a)(4+8a)) \\
= 4(64a^2-8a+48) >0
\end{multline*} and $S$ contains no $(-2)$-curve. 
Then we see that $|\cO_S(16a^2-a+4, 4-8a, 4+8a)|$ is free since there is no $\bP^1$ on $S$
 (cf.\ \cite[Proposition 8.1]{MR0364263}, \cite[Chapter 2, Corollary 3.15(ii)]{MR3586372}). 
 Thus we have $\mu:= \mu' \circ \nu \colon X_1 \rightarrow P(3)$. Let $X_2:= P(3)$. 
 
Let $S_2:= S \subset X_2$ and $S_1 \subset X_1$ be the strict transform of $S$ and $i_j \colon S_j \hookrightarrow X_j$ be the inclusions for $j=1,2$, 
and let $\iota_a:= \iota^a \circ \mu|_{S_1}$.  	
By Corollary \ref{cor:pushoutSNC}, we can construct the push-out $X_0(a)$ of  two closed immersions 
 $i_1$ and $i_2 \circ \iota_a$. For simplicity, we write $X_0:= X_0(a)$. 
Then $X_0$ is a proper SNC variety and fits in the following diagram; 
\[
\xymatrix{
 & S_1 \ar[ld]^{\iota_a} \ar@{^{(}->}[r]^{i_1} & X_1 \ar[dd] \\
S_2 \ar@{^{(}->}[d]^{i_2} & & \\
X_2  \ar[rr]& & X_0. 
}
\]	

The SNC variety $X_0$ satisfies the condition of Theorem \ref{thm:kawamata-namikawa} by the following claim. 

\begin{claim}\label{claim:onX_0} 
\begin{enumerate}
\item[(i)] $X_0$ is $d$-semistable. 
 \item[(ii)]  $\omega_{X_0} \simeq \cO_{X_0}$. 
\item[(iii)] $H^1(X_0 , \cO_{X_0}) =0, H^2(X_0^{\nu}, \cO_{X_0^{\nu}}) =0$, 
where $X_0^{\nu} \rightarrow X_0$ is the normalization. 
\end{enumerate}
\end{claim} 
\begin{proof}[Proof of Claim] 
(i) In order to check the $d$-semistability, we shall show that 
\[
\cO_{S_1}(X_0):= \cN_{S_1/ X_1} \otimes (\iota_a)^*\cN_{S_2/X_2} 
\simeq \cO_{S_1}. 
\] 
Let $\mu_1:= \mu|_{S_1} \colon S_1 \xrightarrow{\simeq} S$ and $\cO_{S_1}(c_1, c_2, c_3):= \mu_1^* \cO_S(c_1, c_2, c_3)$ for 
$c_1, c_2, c_3 \in \bZ$. 
Since we have 
\[
\cN_{S_1/X_1} \simeq \cO_{S_1}(2,2,2) \otimes \cO_{S_1}(- (\sum_{i=1}^{a} f_i + C_a)), 
\]
\[
(\iota_a)^* \cN_{S_2/X_2} \simeq 
(\iota_a)^* \cO_S(2,2,2), 
\]
we obtain  
\begin{multline*}
\cO_{S_1}(X_0) \simeq \cO_{S_1}(2,2,2) \otimes \cO_{S_1} (- (\sum_{i=1}^{a} f_i + C_a)) \otimes (\iota_a)^* \cO_S(2,2,2) \\
 \simeq \cO_{S_1}(16a^2 + 4, 4-8a, 4+8a) 
\otimes \cO_{S_1} (- (\sum_{i=1}^{a} f_i + C_a)) 
 \simeq \cO_{S_1}, 
\end{multline*}
thus we obtain (i). 

\noindent(ii) By $S_i \in | \omega_{X_i}^{-1}|$ for $i=1,2$ and Remark \ref{rem: dualizingtrivial}, we see that $\omega_{X_0} \simeq \cO_{X_0}$
(cf.\ \cite[Remark 2.11]{MR707162}). 

\noindent(iii) The exact sequence 
\[
0 \rightarrow \cO_{X_0} \rightarrow \cO_{X_1} \oplus \cO_{X_2} \rightarrow 
\cO_{X_{12}} \rightarrow 0
\]  
implies $H^1(X_0, \cO_{X_0}) =0$, where $X_{12}:= X_1 \cap X_2 \simeq S_i$ for $i=1,2$.  
We have $H^2(X_0^{\nu}, \cO_{X_0^{\nu}}) =0$ since $X_1$ and $X_2$ are rational. 
\end{proof}

By the above and Theorem \ref{thm:kawamata-namikawa}, there exists a semistable smoothing $\phi_a \colon \cX(a) \rightarrow \Delta^1_{\epsilon}$ of $X_0$ over an open disk 
of a sufficiently small radius $\epsilon >0$. 
Let $X(a)$ be a fiber of $\phi_a$ over $t \neq 0$. Note that we do not specify $t$ and 
all such fibers are diffeomorphic. 
 Let $\cX:= \cX(a)$ and $X:= X(a)$ for simplicity.  
 
Then we have $\omega_{X} \simeq \cO_{X}$ since we have $H^1(X_0, \cO_{X_0}) =0$, $H^1(\cX, \cO_{\cX}) =0$ and, by the diagram 
\[
\xymatrix{
H^1(\cX, \cO_{\cX}^*) \ar[r] \ar[d]^{i_0^*} & H^2(\cX, \bZ) \ar[d]^{\simeq} \\
H^1(X_0, \cO_{X_0}^*) \ar[r] & H^2(X_0, \bZ), 
}
\]
 we see that $i_0^*$ is injective, where $i_0 \colon X_0 \hookrightarrow \cX$ is the inclusion.
% (cf.\ \cite[Theorem (2.18)]{MR0429885}). 

We also check that $H^i(X, \cO_{X}) =0$ for $i=1,2$ by the upper semicontinuity theorem since $\epsilon$ is small. We have the following claim on the Betti numbers $b_i(X)$ of $X$ for $i=1,2$. 
	
\begin{claim}\label{claim:b1b2}
\begin{enumerate}
\item[(i)] We have $H^1(X, \bZ) =0$, thus $b_1(X) =0$. 
\item[(ii)] $b_2(X) = a +3$.  
%(We also see that $\pi_1(X) = \{ 1\}$??)
\end{enumerate}
\end{claim}

\begin{proof}[Proof of Claim] 
(i) By the exponential exact sequence, we have an exact sequence 
\[
H^0(X, \cO_X) \rightarrow H^0(X, \cO_X^*) \rightarrow H^1(X, \bZ) 
\rightarrow H^1(X, \cO_X). 
\]
This implies (i). 

\noindent(ii) Note that $b_2(X_0) = \rk \Pic X_0$ since we can calculate that $H^i(X_0, \cO_{X_0})=0$ for $i=1,2$ as in Claim \ref{claim:onX_0}(iii). 
Moreover, we see that $\rk \Pic X_0 = a +4$ by the exact sequence 
\begin{multline*}
0 \rightarrow H^1(X_0, \cO^*_{X_0}) \rightarrow H^1(X_1, \cO^*_{X_1}) \oplus 
H^1(X_2, \cO_{X_2}^*) \\
\rightarrow H^1(X_{12}, \cO_{X_{12}}^*) \rightarrow 0,  
\end{multline*}
where the surjectivity follows from the explicit description. 
In order to compute $b_2(X)$, we use the Clemens contraction 
$c_t \colon X \rightarrow X_0$ which satisfies that 
$c_t^{-1}(p) \approx S^1$ for $p \in X_{12}$ and $c_t^{-1}(p) = \{{\rm pt} \}$ 
for $p \notin X_{12}$ as in Theorem \ref{fact:Clemenscontraction}. We see that $R^1 (c_t)_* \bZ_{X} \simeq \bZ_{X_{12}}$ since $X_{12}$ is simply connected 
 and that $R^2 (c_t)_* \bZ =0$. By this and the Leray spectral sequence
\[
H^i(X_0, R^j (c_t)_* \bZ) \Rightarrow H^{i+j} (X, \bZ ),  
\]
we see that $b_2(X) = b_2(X_0) -1 = a +3$. 
Indeed, we have \[
H^0(X_0, R^2 (c_t)_* \bZ) =0, \ \ 
H^1(X_0,  R^1(c_t)_* \bZ) = H^1(X_{12}, \bZ) =0, 
\]
\[
H^2(X_0, (c_t)_* \bZ) \simeq
H^2(X_0, \bZ)
\]  and see that the connecting homomorphism 
\[
\bZ \simeq H^0( X_0, R^1 (c_t)_* \bZ) \rightarrow H^2(X_0, (c_t)_* \bZ) \simeq H^2(X_0, \bZ) 
\]
is non-zero by $H^1(X, \bZ) =0$, and its cokernel is $H^2(X, \bZ)$. 
\end{proof}

We compute the topological Euler number $e(X)$ as follows. 

\begin{claim}\label{claim:Eulernumber}
We have $e(X) = -256 a^2 +32a-224$. 
\end{claim}

\begin{proof}[Proof of Claim] We shall use the product formula of topological Euler numbers 
on an oriented fiber bundle (cf.\ \cite[pp.481, Theorem 1]{MR1325242}) and also an additivity formula 
for the Euler number on a complex algebraic variety (cf.\ \cite[pp.95, Exercise]{Fulton_ITV}). 
Note that we have $e(f_i) =0$ for $i=1, \ldots, a$, 
\[e(C_a) = 2-2 g(C_a) = -(C_a^2) =-256 a^2 +32a -192, 
\] and 
an exceptional divisor of a blow-up along a curve  is a $\bP^1$-bundle. Thus 
we see that $e(X_1) = e( P(3)) -256 a^2 +32a -192 =-256 a^2 +32a -184$ by the above two formulas. 
By this and the exact sequence
\[
0 \rightarrow \bZ_{X_0} \rightarrow \bZ_{X_1} \oplus \bZ_{X_2} \rightarrow \bZ_{X_{12}} \rightarrow 0, 
\]  we see that 
\begin{multline*}
e(X_0) = e(X_1) + e(X_2) - e(X_{12}) = (-256 a^2 +32a-184) + 8 - 24 \\ = -256 a^2 +32a-200. 
\end{multline*}
Since $c_t^{-1}(X_{12}) \rightarrow X_{12}$ is an $S^1$-bundle over a K3 surface, we check that 
\[
e(X) = e(X_0) - e(X_{12}) = -256 a^2 +32a-200-24 =-256 a^2 +32a -224  
\]
by the Leray spectral sequence as in Claim \ref{claim:b1b2}(ii). Indeed, $H^i(X_0, R^j (c_t)_* \bZ) =0$ 
except when $j=0,1$. 
\end{proof}

By these claims, we obtain $X$ as described in the statement of Theorem \ref{thm:mainexample}.
\end{proof}	

\subsection{Some properties of $X(a)$} 

First, our examples have the following Hodge-theoretic property.  
	
\begin{rem}\label{rem:HodgeSymmetry}
Note that the Hodge to de Rham spectral sequence degenerates at $E_1$ 
on our Calabi-Yau 3-fold $X:= X(a)$ (cf.\ \cite[Corollary 11.24]{MR2393625}). 
By this and \cite[Theorem 3.3]{MR3867651}, we check that: 

\begin{prop}
 $X$ has unobstructed deformations. 
\end{prop}

We also check that $\dim_{\bC} H^i(X, \Omega^j_X) = \dim_{\bC} H^j(X, \Omega^i_X)$ for any $i, j \in \bZ$ as follows. 

%Let $X_0:= X_0(a)$ as before and $\phi \colon \cX \rightarrow \Delta$ be the smoothing of $X_0$ constructed in the above. 
% Let $\Omega^i_{X^{\dag}_0}:= \Omega^i_{\cX/ \Delta}(\log X_0) |_{X_0}$ be the sheaf which is isomorphic to the $i$-th log differential sheaf on the associated log scheme $X_0^{\dag}$. We see that $R^j \phi_* \Omega^i_{\cX/ \Delta}(\log X_0)$ is locally free similarly as \cite[Theorem 2.7]{Friedman:2017aa} by the $E_1$-degeneration of the spectral sequence. 
%Hence we see that 
%\[H^j(X, \Omega^i_X) \simeq H^j(X_0, \Omega^i_{X_0^{\dag}}).  \] 
We calculate $H^0(X, \Omega^1_X) =0$ by $H^1(X, \bC) =0$ (Claim \ref{claim:b1b2}(i)) and the $E_1$-degeneration. 
We also have $H^0(X, \Omega^2_X) =0$ since we obtain $H^1(X, \cO_X^*) \simeq H^2(X, \bZ)$ 
by considering the exponential exact sequence as in Claim \ref{claim:b1b2}. 
Thus, for $i=1,2$, since we have $H^i(X, \cO_X) =0$, we have the Hodge symmetry on $H^i(X, \bC)$. 

On the direct summands of $H^3(X, \bC)$, we have 
\[
H^j(X, \Omega^{3-j}_{X}) \simeq H^{3-j}(X, \Omega^j_X) 
\]
by the Serre duality and $\omega_X \simeq \cO_X$. 
By these, we have the required equality on $H^3(X, \bC)$. 
Thus we can not judge the non-projectivity of $X$ from the Hodge numbers. 

It might be possible to show the $\partial \bar{\partial}$-lemma on $X$ as in \cite{MR4085665}. 
\end{rem}

It may be interesting to study the fundamental group $\pi_1(X)$, the second Chern class $c_2(X)$, etc. 
For the fundamental group, we have the following. 

\begin{prop}\label{prop:simplyconnected}
$X=X(a)$ is simply connected. 
\end{prop} 
\begin{proof} 
Let $V_i \subset X_i$ be a tubular neighborhood of $X_{12}$ for $i=1,2$ 
which can be regarded as a $\Delta^1$-bundle over $X_{12}$. And let $U_1:= X_1 \cup V_2$ and $U_2:= X_2 \cup V_1$. 
We check that $\pi_1(X_0) = \{1 \}$ by applying van Kampen's theorem to the open covering $X_0 = U_1 \cup U_2$.

Note that $\tilde{X}_{12}:= c_t^{-1}(X_{12}) \rightarrow X_{12}$ is an $S^1$-fibration and, from the homotopy exact sequence, 
we see that $\pi_1(\tilde{X}_{12})$ is a cyclic group generated by the $S^1$-fiber class. 
Let $\tilde{X}_i:= c_t^{-1}(X_i)$ for $i=1,2$ and 
consider a neighborhood $ \tilde{V}_i := c_t^{-1}(V_i) \subset \tilde{X}_i$ of $\tilde{X}_{12}$ for $i=1,2$. 
Let $\tilde{U}_1 := \tilde{X}_1 \cup \tilde{V}_{2}$, $\tilde{U}_2:= \tilde{X}_2 \cup \tilde{V}_1$ and $\tilde{U}_{12}:= \tilde{U}_1 \cap \tilde{U}_2$. 
Note that we can regard $\tilde{V}_1 \cup \tilde{V}_2$ as an annulus bundle over $X_{12}$. By this, we see that $\tilde{U}_i$ is homotopic to $X_i \setminus X_{12}$ for $i=1,2$. 
The following claim is important.    

\begin{claim}\label{claim:pi1complement} 
Let $X_i':= X_i \setminus X_{12}$ for $i=1,2$. Then we have $\pi_1(X_1') = \{1 \}$ and $\pi_1(X_2') \simeq \bZ / 2\bZ$. 
\end{claim}

\begin{proof}[Proof of Claim]
We check that $\pi_1(X_2')$ is  abelian by Nori's result (cf. \cite[Corollary 2.10]{MR732347}) as follows. 
By \cite[p.311, Proposition]{MR1282219}, we see that $\pi_1(X_2') \simeq \pi_1(L^*)$, 
where $L$ is the total space of $\cO_{X_2}(X_{12})$ and $L^* \subset L$ is the complement of the zero section.  
Hence the homotopy exact sequence can be written as 
\[
\pi_1(\bC^*) \rightarrow \pi_1 (X_2') \rightarrow \pi_1(X_2) \rightarrow 1
\]
and this implies that $\pi_1(X_2')$ is abelian by $\pi_1(\bC^*) \simeq \bZ$ and $\pi_1(X_2) = \{1 \}$. 

Thus we compute $\pi_1 (X_2') \simeq H_1(X_2', \bZ) \simeq \bZ / 2 \bZ$ 
by the Gysin long exact sequence
\[
\cdots \rightarrow H_2(X_2, \bZ) \rightarrow H_0(X_{12}, \bZ) \rightarrow H_1 (X_2', \bZ) \rightarrow H_1 (X_2, \bZ) \rightarrow \cdots  
\] 
as in \cite[p.46, (2.13)]{MR1194180}.

Let $E'_j:= E_j \setminus X_{12}$ for $j=1, \ldots, a$ and $F' := F \setminus X_{12}$ be the open subsets of 
$\mu$-exceptional divisors for the blow-up $\mu \colon X_1 \rightarrow X_2$.  
Note that $(X_1')\setminus (E_2' \cup \cdots \cup E'_a \cup F') \simeq X_2'$ since 
$X_1 \rightarrow X_2 = P(3)$ is the blow-up along $f_1, \ldots, f_a$ and the strict transform of $C_a$. 
Note also that $E'_j$ and $F'$ are $\bC$-bundle over the blow-up centers $f_1, \ldots, f_a$ and $C_a$ respectively. 
By these, we compute that $\pi_1(X_1') = \{1 \}$ by van Kampen's theorem as follows. 
Let $W_j' \subset X_1'$ be a tubular neighborhood of $E'_j$ for $j=1, \ldots, a$. 

We compute that 
\[
\pi_1 (X_1' \setminus (E_2' \cup \cdots \cup E'_a \cup F') ) \simeq \pi_1(X_2') \ast_{\pi_1(W_1' \setminus E'_1)} \pi_1(W_1') = \{1 \}  
\] 
as follows: Note that $W_1'$ and $W'_1 \setminus E'_1$ can be regarded as a $\Delta^1$-bundle and a $(\Delta^1)^*$-bundle over $E'_1$, 
where $(\Delta^1)^*:= \Delta^1 \setminus \{0 \}$. 
Then we check that $\pi_1(W_1' \setminus E_1') \rightarrow \pi_1(W_1')$ is surjective and its kernel $K \simeq \bZ$ maps surjectively to $\pi_1(X_2')$ by
 $\mu_* \colon \pi_1(W_1' \setminus E'_1) \rightarrow \pi_1(X_2')$. The latter surjectivity follows from a commutative diagram 
\[
\xymatrix{
 H_0(X_{12}, \bZ) \ar[r]  & H_1(X_2' , \bZ) \ar[r] & 0  \\
 H_0(E_1', \bZ) \ar[r] \ar[u] & H_1(W_1' \setminus E_1', \bZ) \ar[u]
 }
\]
as in \cite[p.46, (2.13)]{MR1194180} since a generator of $H_0(E_1', \bZ)$ is sent to that of $H_1(X_2', \bZ)$. 
Hence we see that $X_1' \setminus (E_2' \cup \cdots \cup E'_a \cup F')$ is simply connected. 

Similarly, we check that the fundamental group does not change if we add divisors $E'_2, \ldots ,E'_a, F' \subset X'_1$.  
In particular, we have $\pi_1 (X_1') = \{1 \}$. 
\end{proof}

By Claim \ref{claim:pi1complement} and the isomorphism 
\[
\pi_1(X) \simeq \pi_1 (\tilde{U}_1) \ast_{\pi_1(\tilde{U}_{12})} \pi_1(\tilde{U}_2) \simeq \pi_1 (X_1') \ast_{\pi_1(\tilde{X}_{12})} \pi_1(X_2'),  
\]  
we obtain $\pi_1(X) = \{1 \}$ since we have the following claim: 

\begin{claim} 
$\pi_1(\tilde{U}_{12}) \rightarrow \pi_1(\tilde{U}_2)$ is surjective. 
\end{claim}

\begin{proof}[Proof of Claim] 
Since we have $\tilde{U}_2 = \tilde{U}_{12} \cup (\tilde{X}_2 \setminus \tilde{X}_{12})$, 
we have 
\[
\pi_1(\tilde{U}_2) \simeq \pi_1(\tilde{U}_{12}) \ast_{\pi_1 (\tilde{V}_2 \setminus \tilde{X}_{12})} \pi_1 (\tilde{X}_2 \setminus \tilde{X}_{12}). 
\]
Since $\tilde{X}_2 \setminus \tilde{X}_{12} \simeq X_2 \setminus X_{12} = X_2'$ and $\tilde{V}_2 \setminus \tilde{X}_{12} \simeq V_2 \setminus X_{12}=: V_2'$, 
it is enough to show the surjectivity of 
\[
(\iota_{V_2'})_* \colon \pi_1(V_2') \rightarrow \pi_1 (X_2'). 
\] 
Note that $V_2'$ is a $(\Delta^1)^*$-bundle over $X_{12}$ and $\pi_1(V_2')$ is a cyclic group,  
thus $\pi_1(X_2')$ and $\pi_1(V_2')$ are abelian. Hence the surjectivity of $(\iota_{V_2'})_*$ follows from the following commutative diagram with exact rows as in \cite[p.46, (2.13)]{MR1194180}: 
\[
\xymatrix{
H_2(V_2, \bZ) \ar[r] \ar[d] & H_0(X_{12}, \bZ) \ar[r] \ar[d] & H_1(V_2 \setminus X_{12}, \bZ) \ar[r] \ar[d] & 0 \\
H_2(X_2, \bZ) \ar[r]  & H_0(X_{12}, \bZ) \ar[r]  & H_1(X_2 \setminus X_{12}, \bZ) \ar[r] & 0. 
}
\]
\end{proof}
 This completes the proof of Proposition \ref{prop:simplyconnected}. 
\end{proof}

An anonymous referee pointed out the following bimeromorphic unboundedness of our examples. 

\begin{rem}\label{rem:nonbimeromorphic}
This is based on the referee's comment. 
Let $a \neq a'$ be positive integers. Then $X(a)$ and $X(a')$ are not bimeromorphic as follows. 

Suppose that they are bimeromorphic and let $\phi \colon X(a) \dashrightarrow X(a')$ be a bimeromorphic map. 
Let $\nu \colon \tilde{X} \rightarrow X(a)$ be a resolution of the indeterminacy of $\phi$ which induces a bimeromorphic morphism 
$\mu \colon \tilde{X} \rightarrow X(a')$. 
We see that $\phi$ is an isomorphism in codimension 1 since $\omega_{X(a)}$ and $\omega_{X(a')}$ are trivial. 
Then we have the push-forward homomorphism $\phi_*:=\nu_* \circ \mu^*: \Pic X(a) \rightarrow \Pic X(a')$. 
We see that $\phi_*$ is an isomorphism since $\phi$ is an isomorphism in codimension one. 
Hence $\rk \Pic X(a) = \rk \Pic X(a')$ and this is a contradiction. 

Hence our examples show the bimeromorphic unboundedness of non-K\"{a}hler Calabi-Yau 3-folds. 
We are not sure whether examples of Clemens--Friedman are bimeromorphically unbounded or not. 
(We do not know whether a bimeromorphic map preserves the Betti number of non-K\"{a}hler Calabi-Yau 3-folds.)
\end{rem}

\subsection{On non-projectivity of $X$}

In this section, we check the non-projectivity of the SNC variety $X_0$ and the Calabi-Yau 3-fold $X$ which are
constructed in Theorem \ref{thm:mainexample}. 

Hironaka \cite{MR0139182} constructed a degeneration of a projective manifold to a proper manifold which is non-projective. 
Thus we can not judge non-projectivity of a general fiber from non-projectivity of a central fiber. 
We use the following lemma to see the non-projectivity of a general fiber of the smoothing. 

\begin{lem}\label{lem:biglinebundle}
Let $\phi \colon \cX \rightarrow \Delta^1$ be a semistable smoothing of a proper SNC variety $\cX_0$ with $\omega_{\cX_0} \simeq \cO_{\cX_0}$ such that some fiber $\cX_t$ of $\phi$ over $t \neq 0$ is a projective Calabi-Yau $n$-fold. 
Assume that $\cX_0$  has only two projective irreducible components $X_1, X_2$ and $X_{12}:= X_1 \cap X_2$ is a simply connected Calabi-Yau $(n-1)$-fold (Note that $\cX_0$ may not be projective).  

Then there exists a big line bundle $\cL_0$ on $\cX_0$ such that 
$h^0(X_i, \cL_0|_{X_i}) >0$ for $i = 1,2$. 
\end{lem}

\begin{comment}
\begin{rem}
$X_{12}$ can be an abelian surface (e.g. (blow-up of $\bP^2) \times E_2 \cup E_1 \times blow-up of \bP^2$ where $E_1, E_2$ are elliptic curves. ). It would also be interesting to study this case.
\end{rem}
\end{comment}

\begin{proof}
We first need the following. 
\begin{claim}\label{claim:H2surjective}
The restriction homomorphism $\gamma \colon H^2(\cX, \bZ) \rightarrow H^2(\cX_t, \bZ)$ is surjective for $t \neq 0$. 
\end{claim}

\begin{proof}[Proof of Claim] 
%Note that $X_{12}:= X_1 \cap X_2$ is also a Calabi-Yau manifold since
By the Clemens contraction $c_t \colon \cX_t \rightarrow \cX_0$ as in Fact \ref{fact:Clemenscontraction}, we may regard $\gamma$ as 
\[
c_t^* \colon H^2(\cX_0, \bZ) \rightarrow H^2(\cX_t, \bZ)
\]
This is surjective since we have 
$
H^1(\cX_0, R^1 (c_t)_* \bZ) =0
$ 
 and $H^0(\cX_0, R^2 (c_t)_* \bZ)=0$. Indeed, we have $R^2 (c_t)_* \bZ =0$ 
 since $\cX_0$ has no triple point, and we see that $R^1 (c_t)_* \bZ \simeq 
 \bZ_{X_{12}}$ since $X_{12}$ is simply connected. Thus we can use the Leray spectral sequence as in Claim \ref{claim:b1b2}. 
\end{proof}

Since we have $h^i(\cX, \cO_{\cX}) =0$ and $h^i(\cX_t, \cO_{\cX_t}) =0$ for $i=1,2$, 
we have $\Pic \cX \simeq H^2(\cX, \bZ)$ and $\Pic \cX_t \simeq H^2(\cX_t, \bZ)$ by the exponential exact sequence. 
Let $\cL_t$ be a very ample line bundle on $\cX_t$. 
By the above and Claim \ref{claim:H2surjective}, there exists a line bundle $\cL$ on $\cX$ such that $\cL|_{\cX_t} \simeq \cL_t$. 
We can lift sections of $\cL_t$ to $\cL$ as follows.   

\begin{claim}\label{claim:extensionsection}
The restriction $H^0(\cX, \cL) \rightarrow H^0(\cX_t,\cL_t)$ is surjective. 
\end{claim}

\begin{proof}[Proof of Claim] 
Since we have an exact sequence
\begin{multline}
H^0(\cX, \cL) \rightarrow H^0(\cX_t, \cL_t) \\ 
\rightarrow H^1(\cX, \cL \otimes \cO_{\cX}(- \cX_t)) 
\xrightarrow{\Phi} H^1(\cX, \cL) \rightarrow H^1(\cX_t, \cL_t), 
\end{multline}
it is enough to show that $\Phi$ is injective. 
We see that $\Phi$ is surjective by $H^1(\cX_t, \cL_t) =0$. 
We also see that $H^1(\cX, \cL)$ is finite dimensional. Indeed, 
$H^1(\Delta^1, \phi_* \cL) =0$ and $H^0(\Delta^1, R^1 \phi_* \cL)$ is finite dimensional 
since $R^1 \phi_* \cL$ is coherent and supported on the origin. By these and $\cO_{\cX}(- \cX_t) \simeq \cO_{\cX}$, 
we see that $\Phi$ is an isomorphism, thus injective. 
\end{proof}
By Claim \ref{claim:extensionsection}, we can choose sections $s_0, \ldots, s_M \in H^0(\cX, \cL)$ 
which lift a basis of $H^0(\cX_t, \cL_t)$. 
Let $Z(s_j) \subset \cX$ be the divisor defined by $s_j$ for $j=0, \ldots, M$. 
% and $B:= \bigcap_{i=0}^M Z(s_i)$ be the base locus. Since $\cL_t$ is very ample, we have $B \subset \cX_0$, thus the divisorial component of $B$ is either one of $X_i$. 
Let \[
m_i:= \min_{j=0, \ldots, M} \{\mult_{X_i}(Z(s_j)) \}
\]
 for $i=1, 2$ and $\cL':= \cL \otimes \cO_{\cX}(- m_1 X_1 - m_2 X_2)$. 
Then we obtain sections $s'_0, \ldots, s'_M \in H^0(\cX, \cL')$ induced by $s_0, \ldots, s_M$ 
whose base locus does not contain $X_1$ and $X_2$. 
Hence there exists  $s' \in H^0(\cX, \cL')$ which does not vanish identically on each $X_i$.  

Now let $\cL_0 := \cL'|_{\cX_0}$. 
Then we have non-zero sections $s'|_{X_i} \in H^0(X_i, \cL_0|_{X_i}) $ for $i=1,2$. 
We also see that $\cL_0$ is big since we have $\cL'|_{\cX_t} \simeq \cL_t$ and check that 
$H^0( \cX, \cL'^{\otimes m}) \rightarrow H^0(\cX_t, \cL_t^{\otimes m})$ is surjective for $m >0$ as Claim \ref{claim:extensionsection}. 
Thus $\cL_0$ has the required property. 
%Hence $s'_0, \ldots, s'_N$ defines a rational map 
%\[
%\Phi_{\cL'} \colon \cX \dashrightarrow \bP^N \times \Delta^1 
%\]
\end{proof}

\begin{rem}
There is a conjecture which states that any smooth degeneration of a projective manifold is Moishezon (\cite[Conjecture 1.1]{MR3127061}, see also \cite{Popovici:2019aa}, \cite{Rao:2019aa}).  
We can also ask whether a semistable degeneration of a projective manifold admits a big line bundle as in Lemma \ref{lem:biglinebundle}. 

%It might be possible to weaken the assumption to that $\cX_t$ is Moishezon. 
%However there is a smooth proper toric variety with no nef and big line bundle (cf.\ \cite{MR2196723}). 
\end{rem}

%\begin{rem}
%It should be possible to remove the assumption that the SNC variety $\cX_0$ has two irreducible components.  
%%if we have the Clemens-Schmid exact sequence in our setting. 
%Note that the only part where we used the assumption is the proof of Claim \ref{claim:H2surjective}. 
%
%%Let $\phi \colon \cX \rightarrow \Delta^1$ be a semistable smoothing of a proper SNC variety $\cX_0$. 
%%If $\cX$ is K\"{a}hler, then we have an exact sequence of mixed Hodge structures (cf.\ \cite[Corollary 11.44]{MR2393625})
%%\begin{equation}\label{eq:ClemensSchmid}
%%H^2(\cX, \bZ) \xrightarrow{i_t^*} H^2(\cX_t, \bZ) \xrightarrow{\nu} H^2(\cX_t, \bZ), 
%%\end{equation}
%%where $i_t \colon \cX_t \hookrightarrow \cX$ is the inclusion and  $\nu$ is the logarithm of the monodromy transformation.  
%%By using this exact sequence, Lee  \cite[Corollary III.2]{MR2708981} proved that $\nu =0$ when $\cX$ is K\"{a}hler and $h^{2,0}(\cX_t) =0$. 
%%In this case, we have the surjectivity of $i_t^*$. 
%%However, we do not know whether we have the same exact sequence as (\ref{eq:ClemensSchmid}) in the case where  
%%the general fiber $\cX_t$ is projective and $\cX$ is not necessarily K\"{a}hler. 
%\end{rem}

We can conclude that $X_0$ and $X$ are both non-projective by the following. 

\begin{prop}\label{prop:nonprojective}
Let $X_0:=X_0(a)$ and $X:= X(a)$ be the SNC variety and the Calabi-Yau $3$-fold constructed in Theorem \ref{thm:mainexample} for $a >0$. Let $\cL_0 \in \Pic X_0$ be a line bundle such that $h^0(X_i, \cL_i) >0$ for $i=1,2$, where $\cL_i:= \cL_0|_{X_i}$. 
Also let $E_j \subset X_1$ for $j=1,\ldots, a$ be the exceptional divisor over the elliptic curve $f_j$. 

Then we have 
\[
\cL_1 \simeq \mu^* \cO_{P(3)}(a_1,0,0) - \sum_{j=1}^{a} b_j E_j, 
\]
\[
 \cL_2 \simeq \cO_{P(3)}(a_1 - \sum_{j=1}^a b_j ,0,0)
 \] for some $a_1 \ge 0$ and $b_j \in \bZ$.  
In particular, $X_0$ does not admit a line bundle as in Lemma \ref{lem:biglinebundle}, thus $X_0$ and $X$ are not projective. 
\end{prop}

\begin{proof}
%Suppose that there exists such a line bundle $\cL_0$ on $X_0$. 
%Let $\cL_i:= \cL_0|_{X_i}$ be the line bundles on $X_i$ for $i=1,2$ induced by $\cL_0$. 
Recall that $\mu \colon X_1 \rightarrow P(3)$ is the blow-up of 
$f_1, \ldots, f_{a}, C_a$ and $X_2= P(3)$, where 
$f_1, \ldots f_{a} \in |\cO_S(1,0,0)|$ and $C_a \in |\cO_S(16a^2-a+4, 4-8a, 4+8a)|$. 
Let $F \subset X_1$ be the $\mu$-exceptional divisor over $C_a$. 
Then we can write 
\[
\cL_1 = \mu^* \cO_{P(3)}(a_1, a_2, a_3) \otimes \cO_{X_1}( - \sum_{j=1}^{a} b_j E_j - cF)
\]
for some integers $a_1, a_2, a_3, b_1, \ldots, b_{a}, c$. 
We can also write 
\[
\cL_2 = \cO_{P(3)}(a_1', a_2', a_3')
\]
for some integers $a'_1, a'_2, a'_3$. 
We see that $a_i, a'_i  \ge 0$ for all $i$ since $\cL_1$ and $\cL_2$ are effective. 
%by intersecting $\cL_1$ with the strict transforms of general fibers of the projections $\bP^1 \times \bP^1 \times \bP^1 \rightarrow \bP^1 \times \bP^1$. We can also check that $b_j >0$ and $c>0$ by intersecting $\cL_1$ with general fibers of projections $E_j \rightarrow f_j$ and $F \rightarrow C_a$. 

Note that $X_0$ is the union of $X_1$ and $X_2$ glued along anticanonical members $S_i \in |{-}K_{X_i}|$ via an isomorphism 
\[
\iota_a:=\iota^a \circ \mu|_{S_1} \colon S_1 \rightarrow S_2.
\] 
Then we have 
\[
\cL_1|_{S_1} \simeq (\iota_a)^* \cL_2|_{S_2} 
\]
and the both sides can be written as follows; 
\[
\cL_1|_{S_1} \simeq \cO_{S_1}(a_1 - \sum_{j=1}^{a} b_j - c(16a^2 -a+4), a_2 - c(4-8a), a_3 - c(4+8a)), 
\]
\begin{multline*}\label{mult:L1S1L2S2}
(\iota_a)^* \cL_2|_{S_2} \simeq \cO_{S_1}(a_1' + a_2'(4a^2-2a) + a_3'(4a^2 +2a),\\
 a_2'(1-2a) +a_3'(-2a), a_2'(2a)+a_3'(1+2a)). 
\end{multline*} 
By comparing the 2nd and 3rd coordinates, we obtain 
\[
a_2 + c (8a-4) = a_2'(1-2a) + a_3' (-2a), 
\]
\[
a_3 - c(8a+4) = a_2' (2a) + a_3' (1+2a).  
\]
These imply that
\begin{equation}\label{eq:1stequation}
c(8a-4) = -a_2 + a_2'(-2a+1) + a_3' (-2a), 
\end{equation}
\begin{equation}\label{eq:2ndequation}
c(8a+4) = a_3 + a_2'(-2a)+ a_3'(-2a-1).  
\end{equation}

Now suppose that one of $a_2, a_3, a_2', a_3'$ is positive. 
By the equation (\ref{eq:1stequation}), we obtain $ c \le 0$. 
$c=0$ is possible only when $a_2 = a_2' = a_3' =0$. 
Then we have $a_3 >0$ and this contradicts (\ref{eq:2ndequation}). 
Hence we obtain $c<0$. 
 Moreover, by (\ref{eq:1stequation}) and (\ref{eq:2ndequation}), we obtain 
\[
0 > 4c = c(8a+4) - \frac{2a}{2a-1}c(8a-4) = a_3 + \frac{2a}{2a-1}a_2 + \frac{1}{2a-1} a_3' \ge 0.  
\]
This is a contradiction and we see that $a_2 = a_3 = a_2' = a_3' =0$. 
This implies $c=0$ and that $\cL_1$ and $\cL_2$ are of the form as in the statement. 
%Suppose that $b_j <0$ for some $j$. Let 
%\[V_1:= \Image (H^0(X_1, \cL_1) \rightarrow H^0(S_1, \cL_1|_{S_1})). \]
%Then the base locus $\Bs |V_1|$ contains $(-b_j)f_j$. 
%Hence we see that $\cL_2':= \cO_{\bP}(a_1'-b_j, a_2', a_3')$ is still effective and the pair $(\cL_1+ b_j E_j, \cL_2')$ induces a big line bundle as in the statement. Thus we may assume that $b_j \ge 0$ for all $j$. 
%
%Suppose that $c <0$. Then, by comparing the 2nd and 3rd coefficients of the above  description, we obtain 
%Then we have 
%\begin{multline*}
%a_3 - c(4+8a) = a_2'(2a) + a_3' (2a+1) \le \frac{2a}{2a-1} (a_2'(2a-1) + a_3'(2a))   \\
%= \frac{2a}{2a-1} (-a_2 - c(8a-4)) \le 
%-c\frac{2a}{2a-1} (8a-4) = -8ca < a_3 -c(4+8a) 
%\end{multline*}
%Thus we have a contradiction and obtain $c \ge 0$. 
%
%Now, again by (\ref{mult:L1S1L2S2}), the 2nd coordinates are positive for $\cL_1|_{S_1}$ and non-posotive for $(\iota^a)^* \cL_2|_{S_2}$ when $a_1, a_2, a_3>0$, and this is a contradiction. 
%We can similarly argue when $a'_1, a'_2, a'_3 >0$ and obtain a contradiction. 
%Hence, if $a >0$, then there is a contradiction on the 2nd coordinates since $a_2 - c(4-8a) >0$ and $a_2'(1-2a) +a_3'(-2a) <0$. 
%Similarly, if $a <0$, then we can see the contradiction on the 3rd coordinates. Thus there is no big line bundle on $X_0$ as in the statement. 
\end{proof}

Furthermore, we compute the algebraic dimension of the very general fiber $X$ as follows. 

\begin{prop}\label{prop:algdimX1}
Let $X=X(a)=\cX_t$ be a smooth fiber of a semistable smoothing $\cX(a) \rightarrow \Delta^1$ over $t \in \Delta^1 \setminus \{0 \}$ as in Theorem \ref{thm:mainexample}. 
Let $a(X)$ be its algebraic dimension. 
\begin{enumerate}
\item[(i)] $X$ admits a surjective morphism $\varphi \colon X \rightarrow \bP^1$ 
whose general fibers are K3 surfaces. 
\item[(ii)] $X$ is not projective. 
\item[(iii)] We have $a(X)=1$ for a very general $t \in \Delta^1$. 
\end{enumerate}
\end{prop}

\begin{proof}
(i) Let $\cH_1:= \mu^* \cO(1,0,0)$ on $X_1$ and $\cH_2:= \cO(1,0,0)$ on $X_2$. These glue to give a line bundle $\cH_0 \in \Pic X_0$ which induces a morphism $X_0 \rightarrow \bP^1$.  
We calculate that $H^1(X_0, \cH_0)=0$.  

Since we have $H^2(\cX, \bZ) \simeq H^2(X_0, \bZ)$, there exists $\cH \in \Pic \cX$ such that $\cH|_{X_0} \simeq \cH_0$. 
We check that $H^1(\cX, \cH) =0$ by $H^1(\cX_0, \cH_0) =0$ and the upper semicontinuity theorem. 
Thus we see that $H^0(\cX, \cH) \rightarrow H^0(\cX_t, \cH_t)$ is surjective  for   $t \in \Delta^1$ 
sufficiently close to $0$.  
Hence the line bundle $\cH_t$ also induces a surjective morphism $\varphi_t \colon X:=\cX_t \rightarrow \bP^1$. 

We check that the general fiber $X_{\lambda}$ of $\varphi_t$ at $\lambda \in \bP^1$ is a K3 surface as follows. 
Let $X_{i, \lambda}$ be the general fiber of the morphism $X_i \rightarrow \bP^1$ induced by $\cH_i$ for $i=0,1,2$. We see that $X_{1, \lambda}$ is isomorphic to a blow-up of $\bP^1 \times \bP^1$ at $16$ points and 
$X_{2, \lambda} \simeq \bP^1 \times \bP^1$. Thus we compute that $H^1(X_{0, \lambda}, \cO) =0$ and this implies that 
$H^1(X_{\lambda}, \cO) =0$ by the upper semi-continuity. Hence $X_{\lambda}$ is a K3 surface. 

\noindent(ii) Suppose that some $\cX_t$ is a projective Calabi-Yau 3-fold. By Lemma \ref{lem:biglinebundle}, 
there exists a big line bundle $\cL_0$ on $\cX_0$ such that $h^0(\cL_0|_{X_i}) >0$ for $i=1,2$. 
However, this does not exist on $X_0(a)$ by Proposition \ref{prop:nonprojective}. 
This is a contradiction and $\cX_t$ is not projective.

\noindent(iii) We note that, for $\cM \in \Pic \cX$, the dimension 
$h^0(\cX_t, \cM_t)$ for $\cM_t:= \cM|_{\cX_t}$ is constant for very general $t \in \Delta^1$ and $h^0(\cX_t, \cM_t) \le h^0(\cX_0, \cM_0)$. This follows from  
the upper semicontinuity theorem and  that $\Pic X_0$ is countable. 

Suppose that $a(X) \ge 2$. Then $X$  admits an effective line bundle $L$ with the Kodaira dimension 
$\kappa(L) \ge 2$.  
Since the restriction homomorphism $H^2(\cX, \bZ) \rightarrow H^2(\cX_t, \bZ)$ is surjective, there exists $\cL \in \Pic \cX$ such that 
$\cL_t:= \cL|_{\cX_t} \simeq L$. 
Then we see that $\kappa(\cL_{t'}) \ge 2$ for very general $t' \in \Delta^1$. 
Hence we obtain  $\kappa(\cL_0) \ge 2$ and $\phi_* \cL \neq 0$. Thus we see that 
$H^0(\cX, \cL) = H^0(\Delta^1, \phi_* \cL) \neq 0$ since $\phi_* \cL$ is coherent. 
Now, for $i=1,2$, let 
\[
m_i := \min \{\ord_{X_i}(s) \mid s \in H^0(\cX, \cL) \setminus \{0 \} \}
\] and 
$\cL':= \cL \otimes \cO_{\cX}(-m_1 X_1 - m_2 X_2)$.  
Then $\cL'$  admits a non-zero section which does not vanish entirely on both  $X_1$ and $X_2$. 
Thus we see that $\cL'_0:=\cL' |_{X_0}$ satisfies the property as in Proposition \ref{prop:nonprojective} and 
$\kappa (\cL'_0) \le 1$. This is a contradiction since $\cL'_0$ should also satisfy $\kappa (\cL'_0) \ge 2$. 

Hence we obtain $a(X) \le 1$. By this and (i), 
we obtain $a(X) =1$. 
\end{proof}

\begin{rem}\label{rem:classC}
We also check that $X= \cX_t$ is not of class $\cC$ for a very general $t$, that is, not bimeromorphic to a K\"{a}hler manifold as follows. 
Suppose that $X$ is of class $\cC$ and has a proper bimeromorphic map $\tilde{X} \rightarrow X$
from a K\"{a}hler manifold $\tilde{X}$. 
Since we also have $H^2(\tilde{X}, \bC) \simeq H^1(\tilde{X}, \Omega^1_{\tilde{X}})$ by $H^2(\tilde{X}, \cO_{\tilde{X}}) =0 = H^0(\tilde{X}, \Omega^2_{\tilde{X}})$, 
we see that $\tilde{X}$ is projective by the Kodaira embedding theorem. 
Thus $X$ is Moishezon and this contradicts Proposition \ref{prop:algdimX1}.

We do not know whether a Calabi-Yau 3-fold of algebraic dimension $\ge 2$ appear 
as some fiber of the smoothing. 
Note that the Moishezon (or class $\cC$) property is not stable under deformation (\cite{MR1107661}, \cite{MR1137099}). 
\end{rem}

\section*{Acknowledgement}
We would like to thank Professor Yoshinori Namikawa for several useful comments. 
We would also like to thank Professors Fr\'{e}d\'{e}ric Campana, Yuji Odaka, Keiji Oguiso, Dmitri Panov, Ken-ichi Yoshikawa and the anonymous referees for valuable comments. 
%The first author was partially supported by JSPS KAKENHI Grant Number JP17K14156.
This work was partially supported by JSPS KAKENHI Grant Numbers JP16K17573, JP17H06127 and JP17K14156.

\bibliographystyle{amsalpha}
\bibliography{sanobibs}

\end{document}